%% file: main.tex
\documentclass{amsart}
\input{macros}

\begin{document}

\title{A new class of maximal triangular algebras}
\author{John Lindsay Orr}
\address{
  Toll House,
  Traquair Road,
  Innerleithen, EH44 6PF,
  United Kingdom
}
\email{me@johnorr.us}
\begin{abstract}
  Triangular algebras, and maximal triangular algebras in particular,
  have been objects of interest for over fifty years. Rich families
  of examples have been studied in the context of many w$^*$- and
  C$^*$-algebras, but there remains a dearth of concrete examples
  in $\bh$. In previous work, we described a family of maximal
  triangular algebras of finite multiplicity. Here, we investigate
  a related family of maximal triangular algebras with infinite multiplicity,
  and unearth new asymptotic structure which these algebras exhibit.
\end{abstract}

\maketitle

\section{Introduction}

Triangular algebras have been studied in a variety of contexts for over fifty years
since Kadison and Singer first introduced the concept of triangularity
in~\cite{KadisonSinger:TrOpAl}. Their initial study was of algebras $\T$ of bounded
operators on a Hilbert space. Such an algebra is said to be \emph{triangular} if
its diagonal subalgebra, $\T\cap \T^*$, is a maximal abelian self-adjoint algebra (\masa) in $\bh$.
In finite dimensions, a \masa\ is just the set of diagonal matrices with respect to
a fixed basis and any matrix algebra containing the \masa\ consists of a span of
matrix units with respect to this basis. The triangularity condition amounts to
$\T$ being precisely the span of matrix units $e_{i,j}$ where $i\preceq j$, as
determined by some partial ordering $\preceq$ of $\{1,2,\ldots,n\}$. The algebra
is then \emph{maximal} as a triangular algebra if and only if the associated
partial order is a linear order.

In infinite dimensions the \masa\ generalizes the algebra of diagonal matrices.
It is not, of course, always associated with a basis, but through the Spectral
Theory, can always be associated with a compact spectral set, and the goal
is to correspondingly associate the triangular algebra with a one-sided action
or partial order on the spectral set. This correspondence
has been the subject of study in a wide range of contexts.
The nest algebras, introduced by Ringrose~\cite{Ringrose:OnSoAlOp}
shortly after the triangular algebras, extended the class of hyperreducible triangular algebras
studied by Kadison and Singer, and proved more tractable than general triangular algebras.
Later authors explored triangular algebras of certain C$^*$-algebras
\cite{PetersPoonWagner:TrAfAl,Power:ClTrSu,MuhlySolel:SuGr}
and of von~Neumann algebras
\cite{MuhlySaitoSolel:CoTrOpAl}
which stimulated a rich body of results by many mathematicians in these
contexts.

Little however is known in detail about maximal triangular algebras on infinite dimensional
Hilbert spaces, where in general the \masa\ is not a Cartan algebra \cite{KadisonSinger:ExPuSt}.
Kadison and Singer \cite{KadisonSinger:TrOpAl} showed that
the lattice of invariant projections of a maximal triangular algebra
must be linearly ordered. They
focused on those maximal triangular algebras whose invariant lattice
is multiplicity free (i.e., has a cyclic separating vector), and showed that in this
case the algebra is determined by its invariant lattice; that it to say, it is a nest algebra.
They called such maximal triangular algebras \emph{hyperreducible}. They also showed,
in contrast to the finite-dimensional case,
that not all maximal triangular algebras are hyperreducible
and indeed that there exist maximal triangular algebras which are irreducible
(i.e., having no non-trivial invariant projections).
Solel~\cite{Solel:IrTrAl} further investigated irreducible triangular algebras.
Poon \cite{Poon:MaTrSu} and, independently, the present author~\cite{Orr:ClTrAl},
showed that in general maximal triangular algebras need not even be norm closed.

But apart from the hyperreducible case
no concrete examples of maximal triangular algebras in $\bh$ were known until \cite{Orr:TrAlIdNeAl}.
There, using techniques derived from
the similarity theory of nests \cite{Davidson:SiCoPeNeAl} it was possible to describe
two classes of non-hyperreducible maximal triangular algebras.
The first of these was based on a tensor-product construction proposed
in \cite{Kadison:TrAlAnCh} (see Theorem~\ref{tensor-product-construction-citation} below).
The other was based on a construction of block operator matrices
(see Theorem~\ref{uniform-algebras-thm} below).
The goal was to study a family of maximal triangular algebras we termed the
\emph{compressible} maximal triangular algebras
(see \cite[Definition 6.1]{Orr:TrAlIdNeAl} and Definition~\ref{compressible-def} below)
which are defined in analogy with the Type~I von~Neumann algebras, and we succeeded
in obtaining
detailed descriptions of most finite-multiplicity compressible maximal triangular algebras
in \cite[Theorem~6.1]{Orr:TrAlIdNeAl}.

The purpose of the present work is to extend the construction of compressible maximal triangular algebras
from finite multiplicity to infinite multiplicity.
In Theorem~\ref{extended-triangular-systems-make-tsrc-triangular} we present a new
construction for infinite multiplicity triangular algebras and in
Theorem~\ref{maximal-algebra-thm} we show when this construction yields maximal triangular algebras.
In Section~\ref{examples-section} we explore the range of examples provided by this
construction, and in Section~\ref{characterizing-section} we present criteria for
recognizing maximal triangular algebras which can be represented in this way.

One feature of this construction is that it exposes a new kind of ``asymptotic triangularity''
condition which appears in infinite multiplicity but not in finite multiplicity.
This is based on the ``liminal seminorms'' introduced in Definition~\ref{liminal-seminorm-def}
and the properties of their support sets seen in
Definition~\ref{definition-of-extended-triangular-system}.
Heuristically these conditions can be thought of as describing the contributions
to the norms of rows and columns which are not localized in individual block matrix
entries, but, rather, are residual in the row or column ``at infinity''.

In this paper we focus on those infinite-multiplicity compressible
maximal triangular algebras which are quite uniform
with respect to this asymptotic behavior, which we term \emph{simple}
algebras, although we also present examples of the more complex, but still tractable,
behavior of non-simple algebras (e.g., Example~\ref{non-simple-uniform-alegbra-example}).
The general case of compressible algebras with infinite multiplicity is still
unclear but this study illustrates the kind of subtleties which arise when
passing from finite to infinite multiplicity. Further work will be needed to
understand the infinite multiplicity case completely.

\section{Preliminaries}

Throughout this paper the underlying Hilbert spaces are always assumed separable.
A \emph{nest} is a set of projections on a Hilbert space which is linearly ordered,
contains $0$ and $I$, and is weakly closed (equivalently, order-complete).
The \emph{nest algebra}, $\alg{\N}$, of a nest $\N$ is the set of bounded operators
leaving invariant the ranges of $\N$. An \emph{interval} of $\N$ is the difference
of two projections $N > M$ in $\N$. Minimal intervals are called \emph{atoms} and
the atoms (if there are any) are pairwise orthogonal. If the join of the atoms is $I$
the nest is called \emph{atomic}; if there are no atoms it is called \emph{continuous}.
See \cite{Davidson:NeAl} for further properties of nest algebras.

Nests have a spectral theory analogous to the spectral theory for self-adjoint
operators \cite{Erdos:UnInNe}. Each nest is unitarily equivalent to a nest constructed
from a triple consisting of a linearly ordered set $X$ which is compact in its order topology,
a finite regular Borel measure $m$, and a measurable multiplicity function
$d:X\rightarrow\NN\cup\{+\infty\}$. Briefly, the construction is as follows:
For each $i\in\NN$ let
$X_i := \{x\in X : d(x) \ge i\}$ and for each $x\in X$ write
$L_x := \{y\in X\ : y \preceq x \}$. Let
$\H_i := L^2(X_i, m)$ and the nest consists of the projections on
$\H := \bigoplus H_i$ corresponding to multiplication by
the characteristic functions of $L_x\cap X_i$ on each $\H_i$.
If the multiplicity function is constant then the nest is said to have
\emph{uniform multiplicity}, the non-zero $\H_i$ are unitarily equivalent,
and $\N$ can be represented as a direct sum of copies of a multiplicity-free
nest. If the nest is continuous we can take $X =[0,1]$ and $m$ to be Lebesgue
measure. This representation also provides each nest with an associated projection-valued
spectral measure
corresponding to multiplication by the characteristic function of a Borel set. When the
nest is continuous we write $E(S)$ for the corresponding spectral measure on the
Borel sets of $[0, 1]$.

We now describe in more detail the two previously-known constructions for maximal triangular
algebras mentioned in the introduction.
The first of these realizes the ``triangular tensor product'' construction
envisioned in \cite{Kadison:TrAlAnCh}:

\begin{theorem}[\cite{Orr:TrAlIdNeAl} Theorem 5.1]\label{tensor-product-construction-citation}
  Let $\N_0$ and $\M_0$ be multiplicity-free nests on the Hilbert spaces
  $\H$ and $\K$ respectively, and let $\N := \N_0\otimes I_\K$ and
  $\M := I_\H\otimes\M_0$.
  Then there is a unique maximal triangular algebra $\T$ satisfying:
  \[
    \alg{\N_0}\otimes\alg{\M_0} \subseteq \T \subseteq \alg{\N_0}\otimes B(\K)
  \]
  Moreover $\T$ is the set of operators $X\in\alg{\N}$ such that (i)~whenever
  $M\in\M$ has both an immediate predecessor and successor in $\M$ then
  \[
    M^\perp X M \in \rinfty
  \]
  and (ii)~whenever $M > M' > M''$ are in $\M$ then
  \[
    M^\perp X M'' \in \rinfty
  \]
\end{theorem}

In the statement of the theorem, $\rinfty$ denotes Larson's Ideal, introduced in
\cite{Larson:NeAlSiTr}, which is the largest off-diagonal ideal of $\alg{\N}$
\cite[Theorem 4.1]{Orr:TrAlIdNeAl}. See also
\cite[Chapter 15]{Davidson:NeAl} for details of tensor products; in the theorem the
tensor products are weakly closed spatial tensor products of the respective algebras.

While Theorem~\ref{tensor-product-construction-citation} will provide us with useful examples,
our main focus in this paper will be on the class of \emph{compressible} maximal
triangular algebras which we introduced in \cite{Orr:TrAlIdNeAl} in analogy with
the Type~I von~Neumann algebras.

\begin{definition}\label{compressible-def}
Let $\T$ be a triangular algebra. Let $\N$ be the lattice of invariant projections
of $\T$, which was shown to be a nest in \cite{KadisonSinger:TrOpAl}. The commutant $\N'$
is a Type~I von~Neumann algebra and so contains a partition of the identity $E_i$
consisting of abelian projections. If such $E_i$ can be chosen so that
$E_i\T|_{E_i\H}$ is maximal triangular for all $i$ we say that $\T$ is \emph{compressible}.
\end{definition}

In \cite{Orr:TrAlIdNeAl} we saw both that the compression of a maximal triangular algebra
to the range of an abelian projection in $\N'$ need not always be maximal, and also that, if such
projections can be found, they can provide a basis for completely describing the algebra.
More precisely, in \cite[Theorem 6.1]{Orr:TrAlIdNeAl} we saw that if $\T$ is
a compressible maximal triangular algebra
and $\N$ has no infinite multiplicity part and satisfies some other mild regularity
conditions on its spectral multiplicity, then $\T$ can be completely described.

In our present study we will go to the other extreme and focus on the case when $\N$ has uniform
infinite multiplicity. (Studying the case of mixed finite and infinite multiplicity is
premature when the full range of infinite multiplicity behavior is not yet understood.)
The starting point for our study will therefore be in analogy with the results from
\cite[Example 6.3]{Orr:TrAlIdNeAl} which present an easily visualized construction for uniform
\emph{finite} multiplicity compressible maximal triangular algebras as finite block operators matrices.
We will give a precise statement of the finite-multiplicity result after first defining the
diagonal seminorm function, which will be another key ingredient of our study:

\begin{definition}
  Let $\N=\{N_t : t\in[0,1]\}$ be a continuous nest where the indexing is
  compatible with the spectral measure (i.e., $N_t = E([0,t])$).
  For $X\in\alg\N$ and $x\in[0,1]$ we define the \emph{diagonal seminorm function}
  $i_x(X)$ by the formula
  \[
    i_x(X) := \inf \{\|(N_t - N_s)X(N_t - N_s)\| \st s<x<t\}
  \]
\end{definition}

\begin{theorem}\label{uniform-algebras-thm}
  Suppose that $\N_0$ is a continuous nest on $\H$
  and $\M_0$ is a finite nest of length $n$ on $\K = \CC^n$,
  and let $\N := \N_0\otimes I_\K$ and
  $\M := I_\H\otimes\M_0$.
  Write $E_i$ for the minimal intervals (atoms) of $\M$. Then every maximal
  triangular algebra $\T$ satisfying:
  \[
    \alg{\N_0} \otimes \M_0' \subseteq \T \subseteq \alg{\N_0}\otimes B(\K)
  \]
  is of the the form:
  \[
    \{X \in \alg\N \st \text{
        for each $1 \le i,j \le n$ and a.e. $x\not\in S_{i,j}$,
        $i_x(E_iXE_j) = 0$
      }
    \}
  \]
  where the sets $S_{i,j}$ ($1 \le i,j \le n$) are Borel subsets of $[0,1]$
  satisfying
  \begin{eqnarray*}
    S_{i,i} &=& [0,1], \\
    S_{i,j} &=& S_{j,i}^c \text{ for $i \not= j$, and}\\
    S_{i,j} \cap S_{j,k} &\subseteq& S_{i,k} \text{ for all $i,j,k$}
  \end{eqnarray*}
\end{theorem}

Note the inclusion condition on $\T$ in the last result holds for any compressible
algebra with finite uniform multiplicity nest, and so there is no loss of generality involved,
just a selection of a fixed representation of the nest.

The key technical result involved in proving the last two theorems
was the \emph{Interpolation Theorem},
which is also a crucial tool in the present work:

\begin{theorem}[Interpolation Theorem; \cite{Orr:TrAlIdNeAl} Theorem 3.1]\label{interpolation-theorem}
  Let $\N$ be a continuous nest indexed as above,
  let $X\in\alg\N$ and, for $a>0$, let
  $S := \{t\in[0,1] : i_t(X) \ge a\}$. Then there are operators $A, B\in\alg\N$
  such that $AXB = E(S)$.
\end{theorem}

Although the substance of this result was proved in \cite[Theorem 3.1]{Orr:TrAlIdNeAl},
it should be noted that the proof there made use of a slightly different diagonal seminorm function
(the $i_x^+$ of Ringrose's \cite{Ringrose:OnSoAlOp}) and that the Interpolation Theorem
based on our present $i_x$'s was given in \cite[Theorem 1.2]{Orr:MaIdNeAl}.

We will use the diagonal seminorm function throughout our results. The following lemma
is routine to prove and captures the key technical properties of the function.

\begin{lemma}\label{properties-of-i}
  For fixed $x\in[0,1]$, $i_x(X)$ is a submultiplicative seminorm on $\alg\N$.
  For fixed $X\in\alg\N$, $i_x(X)$ is an upper semicontinuous function on $[0,1]$.
\end{lemma}

\section{The simple uniform algebras}

In this section we will see the precise definition of the new class of
infinite multiplicity maximal triangular algebras which will be our main object
of study in this paper, and which we term the \emph{simple} uniform algebras
(see Definition~\ref{define-simple-algebra}).

For the rest of this paper, fix $\H$ and $\K$ as separable infinite-dimensional Hilbert spaces.
Let $\N_0$ be a multiplicity-free
continuous nest on $\H$ and $\D_0$ an atomic masa on $\K$.
Let $\N := \N_0 \otimes I_\K$ and $\D := I_\H \otimes \D_0$.
We naturally visualize the elements of $\alg\N$ as infinite block operator matrices
with entries in the continuous nest algebra $\alg{\N_0}$.

The atoms of $\D_0$ are one-dimensional so pick a basis $\e_i$ of $\K$ consisting of
unit vectors in the atoms of $\D_0$. Let $E_{i,j} := I\otimes (\e_i\otimes \e_j^*)$
where we adopt the notation $\alpha\otimes \beta^*$ for the rank-1 operator
$\langle \;.\;, \beta\rangle \alpha$.
Also write $E_i := E_{i,i}$ and
$N_x$ ($x \in [0,1]$) for the nest projections of
$\N$, where the indexing is compatible with the spectral measure (i.e., $N_t = E([0,t])$).

Note that any triangular algebra $\T$ satisfying the inclusion relation
\[
  \alg{\N_0}\otimes\D_0\subseteq\T\subseteq\alg{\N_0}\otimes\B(\K)
\]
is compressible
and so throughout the remainder of this paper, and especially in
Section~\ref{characterizing-section}, we shall focus on triangular
algebras satisfying this relation.

The following definition is just the direct analogue of the sets used in
Theorem~\ref{uniform-algebras-thm}, except with infinite multiplicity. In the proposition
that follows it, we see that these properties alone are not enough to specify
a triangular algebra in the infinite-multiplicity case.

\begin{definition}
  Let $\bfS = (S_{i,j})_{i,j\in\NN}$ be a collection of Borel
  subsets of $[0, 1]$ satisfying:
  \begin{enumerate}
    \item $S_{i,i} = [0,1]$ for all $i\in\NN$
    \item $S_{i,j} \cap S_{j,i} = \emptyset$ for all $i\not=j$ in $\NN$
    \item $S_{i,j} \cap S_{j,k}\subseteq S_{i,k}$ for all $i,j,k\in\NN$
  \end{enumerate}
  Then $\bfS$ is called a \emph{triangular system}.
\end{definition}

\begin{proposition}\label{ts-is-a-triangular-space}
  Let $\bfS$ be a triangular system and let $\T(\bfS)$
  be the set of all $X\in\alg{\N}$ such that
  $i_x(E_i X E_j) = 0$ for each $1 \le i,j < \infty$ and a.e. $x\not\in S_{i,j}$.
  Then $\T(\bfS)$ is a triangular space but is never an algebra; that is to say,
  it is a linear space and $\T(\bfS)\cap\T(\bfS)^*$ is a \masa, but it is not
  closed under multiplication.
\end{proposition}

\begin{proof}
  For each $i,j\in\NN$ and $x\in[0,1]$, the function $X\mapsto i_x(E_iXE_j)$ is a
  norm-bounded seminorm. From this it is routine to see that $\T(\bfS)$ is a norm-closed
  linear space. If $T\in\T(\bfS)\cap\T(\bfS)^*$ then $T\in\alg\N\cap(\alg\N)^*=\N'$. We must
  show $E_iTE_j=0$ for all $i\not=j$, for then $T$ commutes with all $E_i$ and
  so $T\in\N_0'\otimes\D_0$, the diagonal \masa.

  Suppose $E_iTE_j\not=0$ for some
  $i\not=j$. Then $i_x(E_iTE_j)$ is zero almost everywhere outside $S_{i,j}$
  and $i_x(E_jT^*E_i)$ is zero almost everywhere outside $S_{j,i}$. Since these two
  quantities are equal, and $S_{i,j}\cap S_{j,i}=\emptyset$, it follows $i_x(E_iTE_j)$ is zero
  almost everywhere. By Theorem~2.1 of \cite{Orr:TrAlIdNeAl}, $E_iTE_j\in\rinfty$ which
  is a diagonal-disjoint ideal  of $\alg\N$ and yet $E_iTE_j$ belongs to $\N'$, the diagonal
  of $\alg\N$, hence $E_iTE_j=0$.

  To see that $\T(\bfS)$ is not an algebra, we shall fix arbitrary $i,j$ and construct
  operators $X=E_i X$ and $Y=YE_j$ in $\T(\bfS)$ which satisfy $i_x(E_iXYE_j)\ge 1$
  for all $x\in[0,1]$. Let $(s_n, t_n)$ be an enumeration of all the open intervals with
  rational endpoints in $[0,1]$. For each $n$ pick $s_n<x<y<t_n$ and set
  $X_n := \alpha_n \otimes \beta_n^*$ and
  $Y_n := \beta_n \otimes \gamma_n^*$
  where $\alpha_n$, $\beta_n$, and $\gamma_n$ are unit vectors in, respectively,
  the range of
  $(N_x-N_{s_n})E_i$,
  $(N_y-N_x)E_n$, and
  $(N_{t_n}-N_y)E_j$.
  Since each of these ranges is infinite dimensional we can choose the
  $\alpha_n$, $\beta_n$, and $\gamma_n$ inductively to be pairwise orthogonal
  sequences.
  Each $X_n$ and $Y_n$ is in $\alg\N$ since
  $X_n = N_x X_n N_x^\perp$ and $Y_n = N_y Y_n N_y^\perp$, and so
  $X := \sum_n X_n$ and $Y := \sum_n Y_n$ converge strongly to operators in $\alg\N$.
  For any $m,n$ we have $E_m X E_n$ is zero unless $m=i$, in which case
  $i_x(E_i X E_n) = i_x(X_n) = 0$, since $X_n$ is finite rank. Thus
  $i_x(E_m X E_n)$ is zero for all $m,n$ and $x\in[0,1]$ so $X\in\T(\bfS)$ and,
  similarly, $Y \in\T(\bfS)$. On the other hand
  $XY = E_iXYE_j = \sum_n \alpha_n\otimes \gamma_n^*$. Since
  $\|(N_{t_n}-N_{s_n})XY(N_{t_n}-N_{s_n})\| \ge 1$ for all $n$ it follows
  $i_x(E_iXYE_j) \ge 1$ for all $x$ and so $\T(\bfS)$ cannot be closed under
  multiplication.
\end{proof}

Moreover, as we shall see in Proposition~\ref{triangular-contained-in-ts},
every maximal triangular algebra,
$\T$, satisfying $\alg{\N_0}\otimes\D_0 \subseteq \T \subseteq \alg{\N_0}\otimes\B(\K)$
is contained in $\T(\bfS)$ for some triangular system $\bfS$. Thus it makes sense to seek
additional constraints on the elements of $\T(\bfS)$ which will determine a
maximal triangular algebra. In the following two definitions we
introduce the properties, related to ``behaviour at infinity'' of block operator
matrices which enable us to specify triangular algebras.

\begin{definition}\label{liminal-seminorm-def}
  Let $M_i := \sum_{i=1}^n E_i$.
  For $X\in\alg{\N}$, $t\in[0,1]$, and $i, j \in \NN$, define the
  \emph{liminal row seminorm}
  \[
    r_{i,t}(X) := \lim_{k\rightarrow\infty} i_t(E_i X M_k^\perp)
  \]
  and the corresponding \emph{liminal column seminorm}
  \[
  c_{j,t}(X) := \lim_{k\rightarrow\infty} i_t( M_k^\perp X E_j)
  \]
\end{definition}

\begin{remark}
  Despite superficial appearances, the values of $r_{i,t}$ and $c_{j,t}$ do not
  depend on the ordering of the atoms of $\D$.
\end{remark}

The proof of the following basic properties of the liminal seminorms is routine,
and left to the reader.

\begin{lemma}
  For fixed $i,j$ and $t\in[0,1]$, the functions $r_{i,t}$ and $c_{j,t}$
  are seminorms on $\alg\N$.
  For fixed $i,j$ and $X\in\alg\N$, the functions $r_{i,t}(X)$ and $c_{j,t}(X)$
  are upper semicontinuous functions of $t\in[0,1]$.
\end{lemma}

We now add extra properties to the definition of triangular system, which will
enable us to specify a triangular algebra, as seen in the following theorem.

\begin{definition}\label{definition-of-extended-triangular-system}
  Let $\bfS = (S_{i,j})_{i,j\in\NN}$,
  $\bfR = (R_i)_{i\in\NN}$, and $\bfC = (C_j)_{j\in\NN}$ be collections of Borel
  subsets of $[0, 1]$ satisfying:
  \begin{enumerate}
    \item
      \label{definition-of-extended-triangular-system--first-item}
      \label{definition-of-extended-triangular-system--reflexive}
      $S_{i,i} = [0,1]$ for all $i\in\NN$
    \item
      \label{definition-of-extended-triangular-system--antisymmetric}
      $S_{i,j} \cap S_{j,i} = \emptyset$ for all $i\not= j$ in $\NN$
    \item
      \label{definition-of-extended-triangular-system--transitive}
      $S_{i,j} \cap S_{j,k}\subseteq S_{i,k}$ for all $i,j,k\in\NN$
    \item
      \label{definition-of-extended-triangular-system--column}
      $C_i \cap S_{i,j} \subseteq C_j$ for all $i,j\in\NN$
    \item
      \label{definition-of-extended-triangular-system--second-to-last-item}
      \label{definition-of-extended-triangular-system--row}
      $S_{i,j} \cap R_j \subseteq R_i$ for all $i,j\in\NN$
    \item
      \label{definition-of-extended-triangular-system--last-item}
      \label{definition-of-extended-triangular-system--row-column}
      $R_i \cap C_j \subseteq S_{i,j}$ for all $i,j\in\NN$
  \end{enumerate}
  Then the triple $(\bfS, \bfR, \bfC)$ is called an \emph{extended triangular system}.
\end{definition}

\begin{definition}\label{define-tsrc}
  Given collections of Borel sets, $\bfS = (S_{i,j})_{i,j\in\NN}$,
  $\bfR = (R_i)_{i\in\NN}$, and $\bfC = (C_j)_{j\in\NN}$, we shall write
  $\T(\bfS, \bfR, \bfC)$ for the set of all $X\in\alg{\N}$ such that
  \begin{enumerate}
    \item\label{sij-criterion}
      $i_t(E_i X E_j) = 0$ for each $1 \le i,j < \infty$ and a.e. $t\not\in S_{i,j}$
    \item\label{ri-criterion}
      $r_{i,t}(X) = 0$ for each $1 \le i < \infty$ and a.e.\ $t\not\in R_i$
    \item\label{cj-criterion}
      $c_{j,t}(X) = 0$ for each $1 \le j < \infty$ and a.e.\ $t\not\in C_j$
  \end{enumerate}
\end{definition}

\begin{theorem}\label{extended-triangular-systems-make-tsrc-triangular}
  If $(\bfS, \bfR, \bfC)$ is an extended triangular system
  then $\T(\bfS, \bfR, \bfC)$ is a triangular algebra.
\end{theorem}

\begin{definition}\label{define-simple-algebra}
  The algebras $\tsrc$ described in the last theorem are called
  the \emph{simple uniform triangular algebras}.
\end{definition}

\begin{proof}[Proof of Theorem \ref{extended-triangular-systems-make-tsrc-triangular}]
  By the same techniques as Proposition~\ref{ts-is-a-triangular-space}, and since $r_{i,t}$ and
  $c_{j,t}$ are seminorms, $\tsrc$ is a triangular space. It remains to show it is closed
  under multiplication. Let $X, Y \in\tsrc$ and verify criteria
  (\ref{sij-criterion}), (\ref{ri-criterion}), and (\ref{cj-criterion}) for $XY$.

  To verify (\ref{sij-criterion}), fix $i$ and $j$ and consider
  \begin{align*}
    i_t(E_i XY E_j)
      &\le \sum_{k=1}^r i_t(E_iXE_kYE_j) + i_t(E_iXM_r^\perp YE_j) \\
      &\le \sum_{k=1}^r i_t(E_iXE_k) i_t(E_kYE_j) + i_t(E_iXM_r^\perp) i_t(M_r^\perp YE_j)
  \end{align*}
  The terms in the sum are zero almost everywhere outside $S_{i,j}$ and the
  remainder term converges to zero (as $r\rightarrow\infty$) for almost every
  $t$ outside $R_i\cap C_j\subseteq S_{i,j}$.
  Integrate (wrt $t$) over $S^c_{i,j}$ and apply the Dominated Convergence Theorem to the
  limit as $r\rightarrow\infty$ to see that $i_t(E_i XY E_j) = 0$ for almost all
  $t\not\in S_{i,j}$.

  Verify (\ref{ri-criterion}) in the same way by considering the inequality
  \begin{align*}
    i_t(E_i XY M_j^\perp)
      &\le \sum_{k=1}^r i_t(E_iXE_kYM_j^\perp) + i_t(E_iXM_r^\perp YM_j^\perp) \\
      &\le \sum_{k=1}^r i_t(E_iXE_k) i_t(E_kYM_j^\perp) + i_t(E_iXM_r^\perp) \|Y\|
  \end{align*}
  from which, taking $j\rightarrow\infty$,
  \begin{equation*}
    r_{i,t}(XY) \le \sum_{k=1}^r i_t(E_iXE_k) r_{k,t}(Y) + i_t(E_iXM_r^\perp) \|Y\|
  \end{equation*}
  The terms in the sum are zero almost everywhere outside $S_{i,k} \cap R_k\subseteq R_i$
  and the remainder term converges to zero (as $r\rightarrow\infty$) for almost every
  $t$ outside $R_i$. Thus similarly by the Dominated Convergence Theorem
  $r_{i,t}(XY) = 0$ on $R_i^c$.
  The case of criterion~(\ref{cj-criterion}) is analogous.
\end{proof}

We now present a key observation which relates our construction, with families of
measurable sets, to partial orders, something to be expected in the context of
triangular algebras.
Note that an extended triangular system induces a set of partial orders and ``Dedekind cuts''
on $\NN$. More precisely, for each fixed $x\in[0,1]$, we define a partial order on $\NN$
by $i \preceq j$ if $x \in S_{i,j}$ and let $A = \{ i\in\NN \st x\in C_i\}$ and
$B = \{ i\in\NN \st x\in R_i \}$. Note that $A$ is an increasing set, since if $i\in A$
and $i\preceq j$ then $x\in C_i$ and $x\in S_{i,j}$, so that
$x\in C_i\cap S_{i,j}\subseteq C_j$ so that $j\in A$. In the same way, one
sees that $B$ is a decreasing set and that every element in $A$ dominates every element
in $B$ (i.e., $b\preceq a$). Although the pair $(A, B)$ is not exactly a Dedekind cut ---
most importantly it does not always partition $\NN$ --- we shall continue to employ the
terminology because of the unmistakable similarities and the fact that, like a true
Dedekind cut, this pair of sets does indicate the behavior at a missing or virtual point,
in our case the asymptotic point at infinity.

\begin{definition}
  In our context, a Dedekind Cut on a partially ordered set is a pair of subsets $(A, B)$
  such that: $A$ is increasing, $B$ is decreasing, and every element of $A$ dominates
  every element of $B$ (though not strictly).
\end{definition}

We shall see below (Theorem~\ref{maximal-algebra-thm}) that if this induced set of partial
orders and Dedekind cuts is maximal, then $\T(\bfS, \bfR, \bfC)$ is a maximal triangular
algebra. We collect a few simple facts about this ordering/cuts viewpoint
in the following lemmas:

\begin{lemma}\label{properties-of-maximal-extended-triangular-systems}
  Let $(\bfS,\bfR, \bfC)$ be an extended triangular system and for each $x\in[0,1]$ let
  $\preceq_x$ and $(A_x, B_x)$ be the partial order and Dedekind cut induced on $\NN$
  in the context of $x$, as described above. Then $(\bfS,\bfR, \bfC)$ is maximal
  (in the sense that none of the sets in $\bfS$, $\bfR$, or $\bfC$ can be enlarged
  without violating the requirements of a triangular system) iff
  for each $x\in[0,1]$,
    $\preceq_x$ is a linear order,
    $A_x\cup B_x =\NN$,
    and either
      (a)~$A_x$ has no smallest element and $B_x$ has no greatest element, or
      (b)~$\min{A_x}$ and $\max{B_x}$ both exist, and are equal.
\end{lemma}

\begin{proof}
  First suppose $(\bfS,\bfR, \bfC)$ is maximal.

  Fix an arbitrary $x\in[0,1]$.
  Even if $\preceq_x$ were not linear, it
  could at least be extended to a linear order $\preceq'_x$ on $\NN$.
  Enlarge $A_x$ and $B_x$ to
  \[
    A'_x := \{i \st \exists a\in A_x \text{ with } a\preceq'_x i\}
    \text{ and }
    B'_x := \{i \st \exists b\in B_x \text{ with } i\preceq'_x b\}
  \]
  Clearly $A'_x$ is increasing, $B'_x$ is decreasing, and if $a'\in A'_x$ and $b'\in B'_x$
  then there are $a\in A_x$ and $b\in B_x$ with
  \[
    b' \preceq'_x b \preceq_x a \preceq'_x a'
  \]
  and so $(A'_x, B'_x)$ is a Dedekind cut. Enlarge $(\bfS,\bfR, \bfC)$ accordingly
  (i.e., put $x$ in $S'_{i,j}$ whenever $i\preceq'_x j$, etc.) and so by maximality
  each $S'_{i,j} = S_{i,j}$ and so $\preceq_x$ is equal to $\preceq'_x$ and so it linear.

  Now suppose there is a $c\not\in A_x\cup B_x$ for some $x$.
  Since $\preceq_x$ is a linear order it follows that
  $b\prec_x c\prec_x a$ for every $a\in A_x$ and $b\in B_x$. Take
  \[
    A'_x := \{ i \st i \succeq_x c  \}
    \text{ and }
    B'_x := \{ i \st i \preceq_x c  \}
  \]
  Clearly $(A'_x, B'_x)$ is a Dedekind cut and $A_x\subseteq A_x'$ and $B_x\subseteq B_x'$.
  But, having enlarged $(A_x, B_x)$ we can correspondingly enlarge
  $(\bfS,\bfR, \bfC)$ by putting $x$ in some of the $R_i$ and $C_j$,
  contrary to supposition. Thus $A_x \cup B_x = \NN$ for all $x$.

  Finally fix $x$ and consider two cases based on whether or not $A_x\cap B_x$ is empty.
  If it is empty then $A_x$ cannot have a smallest element, for if it did then by maximality
  $B_x$ would have to be $\{i \st i \preceq_x \min{A_x} \}$, which would meet $A_x$.
  Likewise, $B_x$ cannot have a greatest element. On the other hand, if $c\in A_x \cap B_x$
  then by maximality $A_x = \{i \st i \succeq_x c \}$ and
  $B_x = \{i \st i \preceq_x c \}$, hence $\min{A_x} = \max{B_x} = c$.

  To prove the converse, now suppose that $(\bfS,\bfR, \bfC)$ is not maximal.
  Suppose further that for all $x\in[0,1]$, both $\preceq_x$ is linear
  (so that $S_{i,j} = S_{j,i}^c$ for all $i, j$)
  and $A_x\cup B_x = \NN$.
  Since $(\bfS,\bfR, \bfC)$ is not maximal find $(\textbf{S}',\textbf{R}',\textbf{C}')$
  which strictly extends $(\bfS,\bfR, \bfC)$. Since $\textbf{S}$ is maximal,
  one of $\textbf{R}',\textbf{C}'$ must be bigger. Without loss, assume
  $C'_i$ is a proper superset of $C_i$ and let $x\in C'_i\setminus C_i$.
  Thus $i\in A'_x\setminus A_x$ and so $b\preceq_x i \prec_x a$ for all
  $a\in A_x$ and $b\in B_x$. Thus on the one hand case~(b) is impossible.
  On the other, $i$ must belong to one of $A_x$ or $B_x$ by supposition,
  and clearly $i\not\in A_x$, so $i\in B_x$ and so $B_x$ does have a greatest element.
  Thus case~(a) fails too.
  By contrapositive, if for all $x$, the ordering $\preceq_x$ is linear,
  $A_x\cup B_c = \NN$, and
  one of case~(a) or case~(b) holds, then  $(\bfS,\bfR, \bfC)$ must be maximal.
\end{proof}

\begin{lemma}
  Every extended triangular system can be enlarged to a maximal extended triangular system.
\end{lemma}

\begin{proof}
  A routine Zorn's Lemma argument would be enough to see the result, except that we
  must maintain measurability of the sets. For that we will need to enlarge the sets
  in a series of deterministic steps.

  First we shall extend $\preceq_x$ to a linear order for all $x$. Enumerate all the
  pairs $(i_0,j_0)$ in a fixed order and run through them.
  Whenever we come to a pair $(i_0, j_0)$ for which there are $x$ with
  $i_0 \not\preceq_x j_0$ and $j_0 \not\preceq_x i_0$, we extend $\preceq_x$
  to $\preceq'_x$ by declaring $i_0 \preceq'_x j_0$ and, consequently,
  $i \preceq'_x j_0$ for all $i\preceq_x i_0$
  and $j\succeq'_x i_0$ for all $j\succeq_x j_0$.
  This translates to enlarging
  $S_{i_0,j_0}$ to
  $S'_{i_0, j_0} := S_{i_0, j_0} \cup (S_{i_0, j_0}\cup S_{j_0, i_0})^c = S_{j_0, i_0}^c$
  and
  \begin{eqnarray*}
    S'_{i, j_0} &:=& S_{i, j_0} \cup (S_{i, i_0} \cap S'_{i_0, j_0}) \\
    S'_{i_0, j} &:=& S_{i_0, j} \cup (S'_{i_0, j_0} \cap S_{j_0, j})
  \end{eqnarray*}
  Likewise enlarge $A_x$ and $B_x$ to
  \[
    \{i \st a\preceq'_x i \text { for some $a\in A_x$}\}
    \text{ and }
    \{i \st i\preceq'_x b \text { for some $b\in B_x$}\}
  \]
  respectively. This translates to
  \begin{eqnarray*}
    R'_i &:=& \bigcup_{b\in\NN} S'_{i, b} \cap R_b \\
    C'_i &:=& \bigcup_{a\in\NN} C_a \cap S'_{a, i}
  \end{eqnarray*}
  for each $i$.
  At each stage all the sets $(\bfS, \bfR, \bfC)$ grow measurably, and continue to
  be an extended triangular system. So finally replace each set with the
  union of all the intermediate versions and we obtain an extended triangular
  system in which each $\preceq_x$ is linear.

  In a similar way, we shall enlarge $\bfR$ and $\bfC$ to be maximal.
  For each $x$, either $(A_x, B_x)$ is already maximal, or else there is
  an $i_0$ such that
 \begin{equation}\label{enlarge-A}
   A'_x := \{ i \st i \succeq_x i_0\} \supseteq A_x
 \end{equation}
 and
 \begin{equation}\label{enlarge-B}
   B'_x := \{ i \st i \preceq_x i_0\} \supseteq B_x
 \end{equation}
 and $(A'_x, B'_x)$ is a maximal Dedekind cut. Turning this around, we shall
 take each value of $i_0$ in turn and enlarge $\bfR$ and $\bfC$ by those $x$
 for which \ref{enlarge-A} and \ref{enlarge-B} hold. Thus, fix $i_0$ and for each $i$,
 replace $R_i$ with $R_i \cup (S_{i, i_0} \cap (R_{i_0} \cup C_{i_0})^c)$ and $C_i$
 with $C_i \cup (S_{i_0, i} \cap (R_{i_0} \cup C_{i_0})^c)$. Repeat this process
 for each $i_0$, and take the union of the successively enlarged sets.
\end{proof}

The next lemma provides technical results necessary to establish
Theorem~\ref{maximal-algebra-thm}, that the triangular algebras associated with maximal
extended triangular systems are themselves \emph{maximal} triangular algebras.
The lemma will enable us to see that the presence of
operators which violate the constraints of $\tsrc$ leads to violations
of triangularity.

\begin{lemma}\label{row-column-lemma}
  Suppose $X\in\alg{\N}$ and for some $i\in\NN$, $a>0$, and a fixed closed
  $K \subseteq [0,1]$, we have $r_{i,t}(X) \ge a$ for all $t\in K$.
  Let $j\in\NN$. Then there are $A, B \in\alg\N$ satisfying $A=E_iAE_i$ and $B=B E_j$,
  and such that
  \begin{enumerate}
    \item $AXB = E(K) E_{i,j}$
    \item $i_t(B) = 0$ for all $t\not\in K$
    \item $i_t(E_m B) = 0$ for all $t\in[0,1]$ and $m\in\NN$.
  \end{enumerate}
\end{lemma}

\begin{proof}
  Consider the intervals of the form $((p-1)/q, (p+1)/q)$
  for natural numbers $p<q$ and let $(s_n, t_n)$ be an enumeration
  of all such intervals which contain a point of $K$.
  Observe that $t_n-s_n\rightarrow 0$.
  For each $n$, choose $x_n$, $y_n$ and $z_n$ with
  $s_n < x_n < y_n < z_n < t_n$ and $y_n\in K$.
  Taking $Y_n := E_i(N_{z_n} - N_{x_n})X(N_{z_n} - N_{x_n})M_n^\perp$,
  note that since $\|Y_n M_k^\perp\| \ge a$ for all $k$, in particular
  the essential norm of each $Y_n$ is at least $a$.
  Thus, there are orthonormal sequences of vectors
  $\beta_n = E_i(N_{z_n} - N_{x_n})\beta_n$, and
  $\gamma_n = M_n^\perp (N_{z_n} - N_{x_n}) \gamma_n$ such that
  $\langle \beta_m, X\gamma_n \rangle = 0$ for all $m\not=n$,
  and
  $\langle \beta_n, X\gamma_n \rangle > a/2$ for all $n$
  (by, e.g., \cite{Orr:StIdNeAl} Lemma~2.2).
  In addition, pick orthonormal sequences
  $\alpha_n = E_i(N_{x_n} - N_{s_n})\alpha_n$, and
  $\delta_n = E_j(N_{t_n} - N_{z_n})\delta_n$, and
  let
  $A := \sum_n \alpha_n\otimes\beta^*_n$ and
  $B := \sum_n \gamma_n\otimes\delta^*_n$.
  Note that each summand of $A$ satisfies
  $\alpha_n\otimes\beta^*_n = N_{x_n}\alpha_n\otimes\beta^*_n N^\perp_{x_n}$
  so that $A\in\alg\N$. Similarly $B\in\alg\N$.

  Fix $x\not\in K$ and $s<x<t$, and choose $s<s'<x<t'<t$ such that
  $(s', t')$ is a positive distance from $K$. Then
  since $t_n-s_n \rightarrow 0$,
  $(s', t')$ is disjoint from all but finitely many $(s_n, t_n)$ and so
  $(N_{t'} - N_{s'})B(N_{t'} - N_{s'})$ is finite rank. Hence
  $i_x(B) = 0$. Likewise for any fixed $m$, $E_m\gamma_n = 0$ for all
  but finitely many $n$ so that $E_m B$ is finite rank and
  $i_x(E_m B)=0$ for all $x\in[0,1]$ and $m\in\NN$.

  On the other hand $AXB = E_i AXB E_j = \sum_n c_n \alpha_n\otimes\delta_n^*$
  where $c_n > a/2$. Fix $x\in K$ and $s<x<t$ and note there is an $n$
  such that $s<s_n<x<t_n<t$ so that
  $\|(N_t-N_s)AXB(N_t-N_s)\| > a/2$, and so $i_x(AXB) \ge a/2$ for all $x\in K$.
  By the Interpolation Theorem (Theorem~\ref{interpolation-theorem})
  there are $A'$ and $B'$ in $\alg\N$ such that
  $(A'A)X(BB')=E(K)$. Then $(E_iA'A)X(BB'E_{i,j}) = E(K)E_{i,j}$
  and the two other conditions hold for $E_iA'A$ and $BB'E_{i,j}$ by submultiplicativity
  of the diagonal seminorm.
\end{proof}

\begin{theorem}\label{maximal-algebra-thm}
  Let $(\bfS, \bfR, \bfC)$ be a maximal extended triangular system.
  Then $\T(\bfS, \bfR, \bfC)$ is a maximal triangular algebra.
\end{theorem}

\begin{proof}
  Suppose $X\not\in \T(\bfS, \bfR, \bfC)$ and show that the algebra $\A$ generated by $X$
  and $\T(\bfS, \bfR, \bfC)$ is not triangular. Of course if $X\not\in\alg\N$, there is
  an $N\in\N$ such that $N^\perp XN\not=0$. Since $NX^*N^\perp\in\T(\bfS, \bfR, \bfC)$
  this would yield the desired result, so assume $X\in\alg\N$.

  Since $X\not\in\tsrc$, $X$ must fail to satisfy one of the three conditions of membership.
  If it fails the first one then there are $i,j$ such that $i_t(E_i X E_j) \not=0$ on a
  non-null subset of $S_{i,j}^c = S_{j,i}$. By upper semicontinuity of $i_t$ there is
  a closed non-null subset $K$ of $S_{j,i}$ and $a>0$ such that $i_t(E_iXE_j)\ge a$ for
  all $t\in K$. Thus by the Interpolation Theorem
  (Theorem~\ref{interpolation-theorem})
  there are $A=E_iAE_i$ and $B=E_jBE_j$
  in $\alg\N$ such that $AXB=E(K)E_{i,j}$ and, of course, $E(K)E_{j,i}\in\tsrc$ since
  $K\subseteq S_{j,i}$, contradicting triangularity of $\A$.

  Next suppose that $X$ fails the second condition. (The case where it fails the third
  condition is handled analogously.) Then there is an $i$ such that $r_{i,t}(X)\not=0$
  on a non-null subset of $R_i^c$. By upper semicontinuity of $r_{i,t}(X)$ as
  a function of $t$,
  there is a non-null closed subset $K$ of $R_i^c$ and $a>0$ such that $r_{i,t}(X)\ge a$
  for all $t\in K$. There are now two distinct cases to be considered.

  \textbf{Case 1:}
  Suppose that $K$ meets $\bigcup_{j\in\NN} R_j\cap C_j$  in a non-null set. In
  this case, replacing $K$ with a smaller non-null closed set we may assume that
  $K\subseteq R_j\cap C_j$ for some $j$.
  Of course, since $K$ is disjoint from $R_i$, we know $i\not= j$.
  By Lemma~\ref{row-column-lemma} there are
  $A=E_i A E_i$ and $B=BE_j$ in $\alg\N$ such that $AXB=E(K)E_{i,j}$ and in addition
  $c_{j,t}(B) \le i_t(B) = 0$ for all $t\not\in K$ (in particular, for all $t\not\in C_j$)
  and $i_t(E_a B E_b) = 0$ for all $t\in[0,1]$ and all $a,b\in\NN$.
  Thus $A, B\in\tsrc$ and so $E(K)E_{i,j}\in\A$. On the other hand,
  $K\subseteq R_i^c\subseteq C_i$
  (since by Lemma~\ref{properties-of-maximal-extended-triangular-systems},
  $R_i\cup C_i = [0,1]$)
  and so $K\subseteq R_j\cap C_i\subseteq S_{j,i}$ by the properties of
  extended triangular systems, and so $E(K)E_{j,i}\in \tsrc$.
  Thus $E(K)E_{i,j}\in\A\cap\A^*$ but, since $i\not= j$, it does not
  belong to the diagonal \masa\ $\N_0'\otimes\D_0$,
  contradicting triangularity of $\A$.

  \textbf{Case 2:}
  Suppose that $K \cap \bigcup_{j\in\NN} R_j\cap C_j = \emptyset$.
  (Possibly replacing $K$ with a subset to make this intersection empty
  and not just null.)
  For each $t\in K$ the induced Dedekind cut $(A_t, B_t)$ satisfies $A_t\cap B_t=\emptyset$
  and so by Lemma~\ref{properties-of-maximal-extended-triangular-systems},
  $A_t$ has no least element and $B_t$ has no greatest element.
  Since $t\not\in R_i$, this means $i\not\in B_t$ and so $i\in A_t$.
  Since $A_t$ has no least element there is a $j\in A_t$ with $j\preceq_t i$.
  Of course this $j$ depends on $t$ but by decomposing $K$ into a countable union
  over candidate values of $j$ we can find a non-null subset on which the same $j\in A_t$
  satisfies $j\preceq_t i$ for all $t$. Replacing $K$ with a closed non-null subset of
  this, we end up with $K\subseteq S_{j,i}$ and $K\subseteq C_j$.

  By Lemma~\ref{row-column-lemma} there are
  $A=E_i A E_i$ and $B=BE_j$ in $\alg\N$ such that $AXB=E(K)E_{i,j}$ and in addition
  $c_{j,t}(B) \le i_t(B) = 0$ for all $t\not\in K$
  and $i_t(E_a B E_b) = 0$ for all $t\in[0,1]$ and all $a,b\in\NN$.
  Thus $A,B\in\tsrc$ and so $E(K)E_{i,j}\in\A$. However we have arranged that
  $K\subseteq S_{j,i}$ so that $E(K)E_{j,i}\in\tsrc$, again contradicting triangularity
  for $\A$.
\end{proof}

\section{Examples}\label{examples-section}

In this section we will focus on the case where the induced order $\preceq_x$ and Dedekind cuts
$(A_x, B_x)$ are constant on $[0,1]$; in other words, the case where each $S_{i,j}$, $R_i$, and
$C_j$ is either $[0,1]$ or $\emptyset$. It should be borne in mind throughout that all the
behaviours described here can in general be mixed together when non-constant components are used.

To simplify things further, since there will be only one induced order, rather than indexing the
atoms by $\NN$ and adopting a secondary ordering $\preceq_x$, we will index the countable set
of atoms by some other ordered set (e.g., $\ZZ$, $\QQ$, etc.) and work with the natural ordering
from the indexing. Note that if $E_i$ are indexed by $i\in I$,
\[
  r_{i,t}(X) = \inf \{i_t(E_i X M^\perp_F)  \st \text{$F\subseteq I$ is finite}\}
\]
and
\[
  c_{j,t}(X) = \inf \{i_t(M^\perp_F X E_j)  \st \text{$F\subseteq I$ is finite}\}
\]
where $M_F : = \sum_{i\in F} E_i$, so that the values of $r_{i,t}$ and $c_{j,t}$ do not
depend on the ordering of the index set $I$.

\begin{example}\label{example-n}
  Let $E_i$ be indexed by $\NN$ so that $\T(\bfS)$ is the set of bounded infinite block matrices
  with entries from $\alg{\N_0}$ on and above the diagonal, and entries from $\R^\infty_{\N_0}$
  below the diagonal. The only possible maximal Dedekind cuts on $\NN$ are the pairs
  $A = [n, \infty)$, $B = [1, n]$ for $n = 1,2,\ldots$, together with $A = \emptyset$, $B = \NN$.
  The latter case corresponds to $C_i = \emptyset$ and $R_i = [0,1]$ for all $i$. Thus
  there is no asymptotic restriction on the rows in $\tsrc$, but $c_{j,t}(X)=0$ for all $j$
  and almost all $t$. It is easy to see that in fact each $X\in\tsrc$ satisfies
  $M_j^\perp X E_j \in \rinfty$ for all $j\in\NN$, since
  \[
    i_t(M_j^\perp XE_j) \le i_t(M_r^\perp XE_j) + \sum_{k=j+1}^r i_t(E_r X E_j)
  \]
  and integrating with respect to $t$ and applying the Dominated Convergence Theorem
  as $r\rightarrow\infty$ shows that $i_t(M_j^\perp XE_j) = 0$ almost everywhere.
  Thus, taking finite sums of columns we see that in this case $\tsrc$ coincides with
  the maximal triangular algebra of Theorem~\ref{tensor-product-construction-citation}.

  The other cases, however, are new, and consist of bounded infinite block matrices
  as before, with entries from $\alg{\N_0}$ on and above the diagonal and entries from
   $\R^\infty_{\N_0}$ below. However for, for some fixed $n\ge 1$, the first $n$
   rows have no other restrictions, but all rows after that have asymptotically a.e.\ zero
   diagonal support (i.e., $r_{i,t}(x) = 0$). The first $n-1$ columns must be in $\rinfty$
   but the rest have no asymptotic constraint.
\end{example}

\begin{example}\label{example-z}
  Let $E_i$ be indexed by $\ZZ$. In this case $\T(\bfS)$  is the set of bounded
  doubly infinite block operator matrices with entries from $\alg{\N_0}$ on and
  above the diagonal, and entries from $\R^\infty_{\N_0}$ below the diagonal.
  The following maximal Dedekind cuts are possible: (i)~$A=\emptyset$ and $B=\ZZ$;
  (ii)~$A=\ZZ$ and $B=\emptyset$; and (iii)~$A=[n, \infty)$ and $B=(-\infty, n]$ for
  $n\in\ZZ$.

  The first two cases bear a deceptive similarity to the algebras obtained by
  Theorem~\ref{tensor-product-construction-citation} and yet they are not the
  same. Theorem~\ref{tensor-product-construction-citation} gives us the algebra
  $\T$ of all doubly infinite block operator matrices satisfying
  $M_i^\perp X M_i \in\rinfty$ for all $i\in\ZZ$. However in our construction of
  $\tsrc$,  in case~(i), the lower half of each block column is in
  $\rinfty$ (by a similar argument to Example~\ref{example-n}) but the left-hand
  half of each row need not be. In case~(ii) the situation is reversed.

  Moreover
  in $\tsrc$ the asymptotic condition on the rows and the columns is two-sided,
  so that
  \[
    c_{j,t}(X) = 0 \Longleftrightarrow
    \lim_{i\rightarrow +\infty} i_t(M_i^\perp X E_j) =
    \lim_{i\rightarrow -\infty} i_t(M_i X E_j) = 0
  \]
  and similarly for $r_{i,t}(X)$. Thus in case~(i) each column is asymptotically
  zero a.e.\ on the diagonal, both approaching $-\infty$ and approaching
  $+\infty$, and case~(ii) is the same with the roles of rows and columns
  reversed. Case~(iii) is a blend of the two, in which for a fixed $n\in\ZZ$
  the block matrices are asymptotically zero-diagonal (a.e.) on the columns for
  $i<n$ and on the rows for $i>n$. Of course up to re-indexing this is really
  just a single case and we may as well take $n=0$.
\end{example}

\begin{example}\label{example-well-ordered}
  Let $E_i$ be indexed by any well-ordered set $S$.
  If $(A, B)$ is any Dedekind cut where $A$ is non-empty, then $A$ has a smallest
  element and the cut is of the form
  $A = \{i : i\ge a\}$,
  $B = \{i : i \le a\}$. The only other case is $A=\emptyset$, $B=S$.
\end{example}

\begin{example}\label{example-q}
  Let $E_q$ be indexed by $q\in\QQ$. This corresponds to the so-called
  Cantor Nest studied in \cite{Larson:NeAlSiTr}. In this case the maximal Dedekind
  cuts very naturally are either $A=[q, +\infty)$ and $B=(-\infty, q]$ for some
  $q\in\QQ$, or else $A=(\gamma, +\infty)$ and $B=(-\infty, \gamma)$ for an
  irrational $\gamma$. In addition the cases $A=\emptyset$ and $B=\QQ$,
  and $A=\QQ$, $B=\emptyset$ are possible.
\end{example}

As observed at the start of this section, the behaviours of these examples, and
indeed of any other linear orderings of the index set of the atoms $E_i$, can
be blended together at different values of $x\in[0,1]$. Purely for illustrative
purposes, we close this section with the construction of a maximal triangular
algebra which mixes the behaviours of the previous examples in a complex
fashion.

\begin{example}\label{very-complex-behavior-example}
  Let $F_i$ ($i\in\NN$) be a sequence of pairwise disjoint
  measure-dense subsets of $[0,1]$. (Measure-dense means that the set meets
  every non-empty open interval in a non-null set; see
  \cite[Lemma 3.1]{Orr:TrAlIdNeAl} for a construction.)
  Re-index the $F_i$ as $F_i^j$ for $i\in\NN$ and $j\in\{1, 2, 3, 4\}$. For
  $x\in F_i^j$ and $j=1,2,3,4$, let $\preceq_x$ be the ordering of the four
  examples, \ref{example-n}, \ref{example-z}, \ref{example-well-ordered}, and
  \ref{example-q} respectively.
  Now suppose that for each of the examples (indexed by $j$) we have enumerated the
  countable family of maximal Dedekind cuts described in that example
  as $(A^j_i, B^j_i)$ ($i\in\NN$) and adopt that cut for $x\in F_i^j$.
  By Lemma~\ref{properties-of-maximal-extended-triangular-systems},
  this induces a maximal extended triangular system $(\bfS,\bfR, \bfC)$
  and by Theorem~\ref{maximal-algebra-thm}, $\tsrc$ is a maximal triangular
  algebra, with extremely complex internal ordering structure!
\end{example}

\section{Characterizing simple uniform algebras}\label{characterizing-section}

In this section we shall study maximal triangular algebras satisfying
\begin{equation*}
  \alg{\N_0} \otimes \D_0 \subseteq \T \subseteq \alg{\N_0}\otimes B(\K)
\end{equation*}
and will identify conditions under which $\T$ is equal to $\tsrc$ for some extended
triangular system. We shall see in Proposition~\ref{triangular-contained-in-ts} that every
triangular algebra satisfying this condition is associated
with a \emph{nearly triangular system}, that is to say, a family of
sets satisfying all the properties of
Definition~\ref{definition-of-extended-triangular-system}
except the last one. From this, in order to show $\T=\tsrc$,
it will be enough to find conditions
which guarantee that the last property
(i.e., $R_i \cap C_j \subseteq S_{i,j}$ for all $i,j$) is satisfied.
In Theorems~\ref{contains-marginally-zero-operators-thm}
and~\ref{infinite-subsets-algebra-thm} we shall present two
necessary and sufficient criteria for $\T = \tsrc$.

First however we observe that all maximal triangular algebras lying between
$\alg{\N_0} \otimes \D_0$ and $\alg{\N_0}\otimes B(\K)$, whether they are
of the form $\tsrc$ or not, must contain $\rinfty$. This shows that, for
maximal triangular algebras of this type, all of the complexity of behavior
is to be found in the asymptotics at the boundary of the block matrix entries.

\begin{proposition}
  Let $\T$ be a maximal triangular algebra satisfying
  \[
    \alg{\N_0} \otimes \D_0 \subseteq \T \subseteq \alg{\N_0}\otimes B(\K)
  \]
  Then $\rinfty \subseteq \T$.
\end{proposition}

\begin{proof}
  Since $\T$ is a subalgebra of $\alg\N$ and $\rinfty$ is an ideal of $\alg\N$,
  $\T+\rinfty \subseteq \alg\N$ and is an algebra. We shall prove that $\T + \rinfty$
  is triangular and then by maximality $\T + \rinfty = \T$ and the result follows.

  Suppose on the contrary that there is a self-adjoint operator $T+R$ for $T\in\T$
  and $R\in\rinfty$ which is not in the \masa\ $\N_0''\otimes\D_0$. Nevertheless
  $T+R$ is in the diagonal of $\alg\N$, which is $\N' = \N_0''\otimes\B(\K)$ and so there
  must be some $i\not=j$ such that $E_i(T+R)E_j \not= 0$.

  Now $i_t(E_i(T+R)E_j)$ is not almost everywhere zero because if it were then
  by \cite[Theorem 2.1]{Orr:TrAlIdNeAl} $E_i(T+R)E_j$ would be in $\rinfty$ which is
  a diagonal-disjoint ideal of $\alg\N$ and does not have any non-zero elements
  of $\N'$. But $T+R$ is self-adjoint so that, again by
  \cite[Theorem 2.1]{Orr:TrAlIdNeAl}, $i_t(R) \eqae 0$, and so
  \[
    i_t(E_iTE_j) \eqae i_t(E_i(T+R)E_j) = i_t(E_j(T+R)E_i) \eqae i_t(E_jTE_i).
  \]
  Thus there is an $a>0$ and a non-null set $K$ such that
  $i_t(E_iTE_j)$ and $i_t(E_jTE_i)$ are both at least $a$ on $K$.
  By the Interpolation Theorem, there are $A, B, C$, and $D$ in $\alg\N$
  such that
  \[
    AE_iTE_jB = CE_jTE_iD = E(K) \not= 0
  \]
  Thus $E_i AE_iTE_jB E_{i,j} = E(K)E_{i,j}$
  and $E_j CE_jTE_iD E_{j, i} = E(K)E_{j, i}$. Since the operators
  $E_i AE_i$, $E_jB E_{i,j}$, $E_j CE_j$, and $E_iD E_{j, i}$
  all belong to $\alg{\N_0} \otimes \D_0 \subseteq \T$,
  it follows that $E(K)E_{i,j}$ and $E(K)E_{j, i}$ belong to $\T$,
  contradicting triangularity.
\end{proof}

The following two lemmas provide necessary technical tools for the theorems of
this section. Remark~\ref{interval-family-remark} below describes how these lemmas
are used in the sequel.

\begin{lemma}\label{sub-sum-lemma}
  Let $A_i$ be a bounded sequence of operators such that for each $i$
  $\lim_{j\rightarrow\infty}\| A_iA_j^* \| =
    \lim_{j\rightarrow\infty}\| A^*_iA_j \| = 0$.
  Then there is a subsequence $k(i)$ such that
  $\sum_{i=1}^\infty A_{k(i)}$ converges strongly.
  Moreover, given a sequence of infinite subsets $S_i$ of $\NN$,
  we can choose the subsequence so that each $k(i)\in S_i$.
\end{lemma}

\begin{proof}
  Fix $\alpha_i > 0$ such that $\sum_i \alpha_i < \infty$.
  The result clearly follows if we can construct $k(i)$ and
  two sequences of pairwise orthogonal projections $P_i$, $Q_i$
  such that
  \[
    \| A_{k(i)} - P_iA_{k(i)}Q_i \| \le \alpha_i
  \]
  for all $i$. We shall do this inductively and,  to ease the induction step,
  we shall add the hypothesis that each $P_i$ and $Q_i$ satisfies
  $\lim_{j\rightarrow\infty}\|P_iA_j\|=\lim_{j\rightarrow\infty}\|A_jQ_i\|=0$.

  To start the induction, pick $k=k(1)$ in $S_1$ and take
  \[
    P_1 := E_{|A_k^*|}( \, [\alpha_1/2, \infty) \, ) \text{ and }
    Q_1 := E_{|A_k|}( \, [\alpha_1/2, \infty) \, )
  \]
  where $E_H$ denotes the spectral measure on $\RR$ for a self-adjoint
  operator $H$.
  Then clearly
  $\|P_1^\perp A_k\| = \|P_1^\perp |A_k^*|\,\| \le \alpha_1/2$ and
  $\|A_k Q_1^\perp\| = \| \, |A_k| Q_1^\perp \| \le \alpha_1/2$, so that
  \begin{equation}\label{bottom-corner-approx}
    \| A_k - P_1A_kQ_1 \|
      \le \| P_1^\perp A_k \| + \| A_k Q_1^\perp \|
      \le \alpha_1
  \end{equation}
  Notice also that
  $A_jQ_1 = A_jA_k^*A_k f(|A_k|)$ where $f(t)=1/t^2$ on $[\alpha_1/2,\infty)$
  and zero elsewhere. Thus $\lim_{j\rightarrow\infty}\|A_jQ_1\|=0$ and by
  a similar argument $\lim_{j\rightarrow\infty}\|P_1A_j\|=0$,
  so the induction hypotheses hold for $k=k(1)$.

  Next suppose that $k(i)\in S_i$, along with pairwise orthogonal $P_i, Q_i$,
  have been found to satisfy the induction hypotheses for $1\le i < n$.
  Let
  $P := (P_1 +\cdots + P_{n-1})^\perp$ and
  $Q := (Q_1 +\cdots + Q_{n-1})^\perp$.
  For each $j$ write $A'_j:=PA_jQ$.
  Clearly $P^\perp A_j$ and $A_j Q^\perp$ converge to zero in norm
  as $j\rightarrow\infty$, so
  $\| A_k - A'_k\| = \| A_k - PA_kQ \| < \alpha_n/2$
  for all sufficiently large $k$. Pick one such $k := k(n)\in S_n$.
  Let
  \[
    P_n := E_{|{A'_k}^*|}( \, [\alpha_n/4, \infty) \, ) \text{ and }
    Q_n := E_{|A'_k|}( \, [\alpha_n/4, \infty) \, )
  \]
  where clearly $P_n\le P$, $Q_n\le Q$ and, as with (\ref{bottom-corner-approx}),
  \begin{equation*}
    \| A'_k - P_nA_kQ_n \|
      = \| A'_k - P_nA'_kQ_n \|
      \le \| P_n^\perp A'_k \| + \| A'_k Q_n^\perp \|
      \le \alpha_n/2
  \end{equation*}
  Thus
  $\| A_k - P_nA_kQ_n \| \le \alpha_n$.

  Also, $\|A_j Q_n\| \le \|P^\perp A_j\| + \|PA_jQ\,Q_n\|$. The first term
  converges to zero so we must show $\|A'_jQ_n\|$ converges to zero.
  As before,
  $A'_j Q_n = A'_j {A'_k}^*A'_k f(|A'_k|)$ where
  $f(t)=1/t^2$ on $[\alpha_n/4, \infty)$ and zero elsewhere.
  Since
  \[
    \|A'_j{A'_k}^*\|
      \le \|A_j Q A_k^*\|
      \le \|A_j A_k^*\| + \sum_{i=1}^{n-1} \|A_jQ_i\|\|A_k\|
  \]
  and all the terms on the right converge to zero, it follows
  that $\lim_{j\rightarrow\infty}\|A'_jQ_n\|=0$
  and so $\lim_{j\rightarrow\infty}\|A_jQ_n\|=0$.
  By similar reasoning,  $\lim_{j\rightarrow\infty}\|P_nA_j\|=0$,
  which completes the induction.
\end{proof}

\begin{lemma}\label{linking-lemma}
  Let $a>0$ and let $(A_i)$, $(B_i)$ and $(D_i)$ be sequences of operators satisfying
  $\|A_iD_iB_i\|>a$ for all $i\in\NN$. Suppose further that
  $(A_i)$ converges strongly to zero,
  $(B_i)$ is bounded and
  the $(D_i)$ are compact and converge strong-* to zero.
  Then given $\e>0$ we can pick a subsequence $k(i)$ such that
  $D := \sum_{i=1}^\infty D_{k(i)}$ converges strongly,
  and
  \[
    \|A_{k(i)}DB_{k(i)}\| > (1-\e)a
  \]
  for all $i\in\NN$.
  Moreover, given a sequence of infinite subsets $S_i$ of $\NN$,
  we can choose the subsequence so that each $k(i)\in S_i$.
\end{lemma}

\begin{proof}
  We shall choose the subsequence $k(i)$ inductively. In order to
  meet the constraints on $k(i)$ and $S_i$, we fix a sequence
  $m(i)$ of natural numbers which takes every value
  in $\NN$ infinitely many times. When choosing $k(i)$, we shall
  ensure that each $k(i)\in S_{m(i)}$ so that ultimately each
  $k^{-1}(S_i)$ will be infinite.
  After the sequence $k(i)$ has been chosen we shall
  apply Lemma~\ref{sub-sum-lemma}
  to $D_{k(i)}$ and $k^{-1}(S_i)$ to get a subsequence $D_{k(l(i))}$
  such that $\sum_i D_{k(l(i))}$ converges strongly and each
  $l(i)\in k^{-1}(S_i)$, so that $k(l(i))\in S_i$.

  By Banach-Steinhaus, the sequences are all bounded in
  norm; let $K$ bound all of them.
  To choose $k(i)$, for each $i$ we pick a unit vector $\xi_i$
  such that $\|A_iD_iB_i\xi_i\| > a$. We start the induction with
  an arbitrary $k(1)$ in $S_{m(1)}$ and then suppose the first
  $n-1$ values have already been chosen. Because the $D_i$ are
  compact, each $A_k D_{k(j)}$ converges to zero in norm as
  $k\rightarrow\infty$ and so
  \begin{equation}\label{why-sum-up-to-n-1-is-small}
    \sum_{j=1}^{n-1} \|A_k D_{k(j)} B_k\|
      \le K \sum_{j=1}^{n-1} \|A_k D_{k(j)}\|
      < \frac{\e a}{2}
  \end{equation}
  for all sufficiently large $k$.
  Likewise, for all sufficiently large $k$,
  \begin{equation}\label{why-sum-from-n-plus-1-is-small}
    \max_{1\le j < n}\|D_k B_{k(j)}\xi_{k(j)}\| < \frac{\e a}{2^{n+1} K}
  \end{equation}
  Choose $k(n)\in S_{m(n)}$ to satisfy both of
  these.

  With the subsequence $k(i)$ chosen in this way, observe
  that $(D_i)$ satisfies the hypotheses of Lemma~\ref{sub-sum-lemma} and
  so we can find a subsequence $D_{k(l(i))}$ of $D_{k(n)}$ such that
  $D := \sum_i D_{k(l(i))}$ converges strongly and, as outlined in the
  first paragraph, $k(l(i))\in S_i$ for all $i$. Now write $k' := k\circ l$
  and observe that
  \begin{align*}
    \|A_{k'(n)} DB_{k'(n)}\|
      &\ge \|A_{k'(n)} DB_{k'(n)} \xi_{k'(n)}\| \\
      &\ge \|A_{k'(n)} D_{k'(n)}B_{k'(n)} \xi_{k'(n)}\| \\
      &\qquad  - K \sum_{j=1}^{n-1} \| A_{k'(n)} D_{k'(j)} \| \\
      &\qquad  - K \sum_{j=n+1}^\infty \| D_{k'(j)}B_{k'(n)} \xi_{k'(n)}\|
  \end{align*}
  Recall we chose $k(n)$ so that
  \[
    \sum_{j=1}^{n-1} \|A_{k(n)} D_{k(j)}\| < \frac{\e a}{2K}
    \text{ and }
    \max_{1\le j < n} \|D_{k(n)} B_{k(j)}\xi_{k(j)}\| < \frac{\e a}{2^{n+1}K}
  \]
   and so, substituting $l(n)$ for $n$,
  \[
    \sum_{j=1}^{l(n)-1} \|A_{k'(n)} D_{k(j)}\| < \frac{\e a}{2K}
    \text{ and }
    \max_{1\le j < l(n)} \|D_{k'(n)} B_{k(j)}\xi_{k(j)}\| < \frac{\e a}{2^{l(n)+1}K}
  \]
  Since clearly $l(1),\ldots, l(n-1)$ are found among $1,2,\ldots,l(n)-1$
  it follows that
  \[
    \sum_{j=1}^{n-1} \|A_{k'(n)} D_{k'(j)}\| < \frac{\e a}{2K}
  \]
  and
  \[
    \max_{1\le j < n} \|D_{k'(n)} B_{k'(j)}\xi_{k'(j)}\|
      < \frac{\e a}{2^{l(n)+1}K}
      \le \frac{\e a}{2^{n+1}K}
  \]
  From this it is clear that
  $\|A_{k'(n)}DB_{k'(n)}\| > a - \e a/2 - \e a/2 = (1-\e)a$
  (exchanging the roles of $j$ and $n$ in the last term).
\end{proof}

\begin{remark}\label{interval-family-remark}
  In subsequent results, we will apply the following technique
  when we make use of Lemma~\ref{linking-lemma}.
  Suppose that $K$ is a non-null closed subset of $(0, 1)$.
  By the Cantor-Bendixon Theorem, we can find a perfect subset $K'$ of $K$
  which differs from $K$ by a countable, and hence null, set.
  Consider all the intervals of the form $((p-1)/q, (p+1)/q)$
  for natural numbers $p<q$ and let $(s_n, t_n)$ be an enumeration of
  all such intervals which contain a point of $K'$.
  Every element of $K'$ (and almost every element of the original $K$)
  belongs to an interval $(s_n, t_n)$ and every interval contains
  infinitely many points of $K$.
  For any $\e>0$,
  only finitely many $t_n-s_n$ are greater than $\e$ and so $t_n-s_n \rightarrow 0$
  and $N_{t_n} - N_{s_n}$ converges to zero strongly.
  Let $S_n := \{ m : (s_m, t_m) \subseteq (s_n, t_n)\}$
  and observe each $S_n$ is infinite since $(s_n, t_n)$
  contains infinitely many points of $K$ and so, given any $N$, we can
  choose $q$ large enough that there are at least $N$ pairwise disjoint
  intervals of the form $((p-1)/q, (p+1)/q)$ which contain elements
  of $(s_n, t_n)\cap K$. In arguments below we will apply
  Lemma~\ref{linking-lemma}, using operators $A_n$ and $B_n$
  which are derived from expressions involving $N_{t_n} - N_{s_n}$,
  and we will obtain subsequences $k(n)$ satisfying $k(n)\in S_n$ for
  all $n$. We will then focus on a fixed $x\in K'$ and interval $(s, t)$
  containing $x$. Then, clearly we can find $n$ such that
  $(s_n,t_n)\subseteq(s,t)$ so that also
  $(s_{k(n)},t_{k(n)})\subseteq(s,t)$ and
  $N_{t_{k(n)}} - N_{s_{k(n)}} \le N_t - N_s$. This will enable us to
  apply norm estimates involving $k(n)$ obtained from
  Lemma~\ref{linking-lemma} to general intervals containing $x\in K'$.
\end{remark}

\begin{proposition}\label{triangular-contained-in-ts}
  Let $\T$ be a triangular algebra satisfying
  \[
    \alg{\N_0} \otimes \D_0 \subseteq \T \subseteq \alg{\N_0}\otimes B(\K)
  \]
  Then there are families of Borel sets
  $\bfS = (S_{i,j})$,
  $\bfR = (R_i)$, and
  $\bfC = (C_j)$ satisfying properties
  (\ref{definition-of-extended-triangular-system--first-item}) to
  (\ref{definition-of-extended-triangular-system--second-to-last-item})
  of Definition~\ref{definition-of-extended-triangular-system} such that
  $\T\subseteq\tsrc$.
  Such a collection of sets will be called a \emph{nearly triangular system} for $\T$.
\end{proposition}

\begin{proof}
  Note that it's enough to prove each of the relations of
  Definition~\ref{definition-of-extended-triangular-system} to within a null
  set. For, after sets have been found to satisfy the relations to within null sets,
  simply form the union of all the excess sets
  $S_{i,j}\cap S_{j,i}$,
  $(S_{i,j}\cap S_{j,k})\setminus S_{i,k}$,
  $(C_i \cap S_{i,j})\setminus C_j$, and
  $(S_{i,j} \cap R_j)\setminus R_i$ for $i,j,k\in\NN$ and remove this
  null set from each of the individual sets $S_{i,j}$, $R_i$, and $C_j$.
  Now all the required properties
  hold exactly, except possibly $S_{i,i} = [0,1]$, and so finally
  enlarge the $S_{i,i}$ by a null set to equal $[0,1]$, which does
  not alter the validity of the other relations.

  Recall from elementary Measure Theory that in any $\sigma$-finite measure
  space, given a (not necessarily countable)
  collection of sets $S_\alpha$ ($\alpha\in A$), we can find a measurable set
  $S := \essunion_{\alpha\in A} S_\alpha$ called the \emph{essential union} of the
  family having the property that
  $S_\alpha\setminus S$ is null for all $\alpha$, and for any measurable $K$,
  if $K\cap S_\alpha$ is null for all $\alpha$ then $K\cap S$ is also null.
  The set is not unique but is unique to within a null set, and we shall
  assume an arbitrary choice has been made to assign a concrete value to
  $\essunion_{\alpha\in A} S_\alpha$.

  For each $i,j\in\NN$ let
  \begin{align*}
    S_{i,j} &:= \essunion_{X\in\T, a>0} \{t \st i_t(E_iXE_j) \ge a \} \\
    R_{i}   &:= \essunion_{X\in\T, a>0} \{t \st r_{i,t}(X) \ge a \} \\
    C_{j}   &:= \essunion_{X\in\T, a>0} \{t \st c_{j,t}(X) \ge a \}
  \end{align*}
  Clearly $\T\subseteq\tsrc$; it remains to
  show $\bfS, \bfR, \bfC$ is a nearly triangular system, at least to within null
  sets.

  Property~(\ref{definition-of-extended-triangular-system--reflexive})
  is trivial, since $\T$ is unital.

  Suppose Property~(\ref{definition-of-extended-triangular-system--antisymmetric})
  does not hold. Then for some $i\not=j$, $S_{i,j}\cap S_{j,i}$ is non-null and so
  there must be $X,Y\in\T$ and $a>0$ such that
  $i_t(E_iXE_j)$ and $i_t(E_jYE_i)$ are at least $a$ on a non-null set $K$.
  Thus  $i_t(E_iXE_{j,i})\ge a$ on $K$
  (since $i_t(E_iXE_{j,i}) \ge i_t(E_iXE_{j,i}E_{i,j}) = i_t(E_iXE_j)$)
  and so by the Interpolation Theorem (Theorem~\ref{interpolation-theorem})
  there are $A, B$ in $\alg\N$ such that $AE_iXE_{j, i}B=E(K)$. Thus
  $E_iAE_iXE_{j, i}BE_{i,j}=E(K)E_{i,j}$ which must be in $\T$, since
  $E_iAE_i$ and $E_{j, i}BE_{i,j}$ are in $\alg{\N_0}\otimes\D_0$. But
  by the same argument applied to $Y$, $E(K)E_{j,i}$ must be $\T$, which
  would contradict triangularity.

  Suppose Property~(\ref{definition-of-extended-triangular-system--transitive})
  does not hold. Then for some $i,j,k$, $S_{i,j} \cap S_{j,k}\setminus S_{i,k}$
  is non-null and there must be $X,Y\in\T$ and $a>0$ such that
  $i_t(E_iXE_j)$ and $i_t(E_jYE_k)$ are greater than $a$ on a non-null set $K$
  which is disjoint from $S_{i,k}$.
  Choose intervals $(s_n, t_n)$ and subsets $S_n$ of $\NN$ as in
  Remark~\ref{interval-family-remark}. Take
  $A_n := (N_{t_n}-N_{s_n})E_iXE_j$ and
  $B_n := E_jYE_k(N_{t_n}-N_{s_n})$.
  Note that $t_n-s_n\rightarrow 0$, so that $A_n$ and $B_n$ converge
  strongly to $0$.
  For each $n$ pick $s_n<x<y<z<t_n$ where $x$ and $z$ are in $K$.
  Since
  $\|(N_y-N_{s_n})E_iXE_j(N_y-N_{s_n})\|$ and
  $\|(N_{t_n}-N_y)E_jYE_k(N_{t_n}-N_y)\|$
  are both greater than $a$, we can find a rank-1 contraction
  $D_n= (N_y-N_{s_n})E_jD_nE_j(N_{t_n}-N_y)$ such that
  $\|A_nD_nB_n\|> a^2$. Thus by Lemma~\ref{linking-lemma} there is
  a subsequence $k(n) \in S_n$ such that $D := \sum_{n=1}^\infty D_{k(n)}$
  converges strongly, and
  \[
    \|(N_{t_{k(n)}}-N_{s_{k(n)}})E_iXDYE_k(N_{t_{k(n)}}-N_{s_{k(n)}})\|\ge a^2/2.
  \]
  By Remark~\ref{interval-family-remark}, for almost every $x\in K$,
  if $(s, t)$ contains $x$ then
  we can find an $n$ such that $(s_{k(n)}, t_{k(n)})\subseteq (s,t)$
  so that also
  $\|(N_t-N_s)E_iXDYE_k(N_t-N_s)\| \ge a^2/2$.
  Thus $i_x(E_iXDYE_k)\ge a^2/2$ for almost all $x\in K$.
  However each $D_n\in E_j\alg\N E_j$, so that
  $D\in E_j\alg\N E_j\subseteq \T$ and so $XDY\in\T$. This contradicts
  the fact that $K$ is disjoint from $S_{i,k}$ and so
  Property~(\ref{definition-of-extended-triangular-system--transitive})
  must hold.

  The proofs of Properties~(\ref{definition-of-extended-triangular-system--column})
  and~(\ref{definition-of-extended-triangular-system--row}) are similar
  to each other, so we present only the first.
  Suppose Property~(\ref{definition-of-extended-triangular-system--column})
  does not hold. Then there are $i$ and $j$ such that
  $C_i \cap S_{i,j} \setminus C_j$ is non-null. As before, find
  $X,Y\in\T$, $a>0$, and a non-null set $K$ disjoint from $C_j$
  on which $c_{i,t}(X)$ and $i_t(E_iYE_j)$ are greater than $a$.
  As before, choose intervals $(s_n, t_n)$ and subsets $S_n$ of $\NN$ according to
  Remark~\ref{interval-family-remark}. As usual let $M_n := E_1 + \cdots + E_n$.
  Let
  $A_n := M_n^\perp(N_{t_n}-N_{s_n})XE_i$ and
  $B_n := E_iYE_j(N_{t_n}-N_{s_n})$.
  Note that $A_n, B_n\rightarrow 0$ strongly. For each $n$, as in the previous case,
  we can find $s_n<x<y<z<t_n$ with $x, z\in K$. Thus
  \[
    \|(N_y-N_{s_n})A_n(N_y-N_{s_n})\|
    \text{ and }
    \|(N_{t_n}-N_y)B_n(N_{t_n}-N_y)\|
  \]
  are both greater than $a$ and so there is a finite-rank contraction
  $D_n = (N_y-N_{s_n})E_i D_n E_i(N_{t_n}-N_y)$ such that
  $\|A_nD_nB_n\|>a^2$. Clearly $D_n$ converges strong-* to zero.
  By Lemma~\ref{linking-lemma} we find $k(n)\in S_n$ such that
  $D := \sum_{n=1}^\infty D_{k(n)}$
  converges strongly, and
  \[
    \|M^\perp_{k(n)}(N_{t_{k(n)}}-N_{s_{k(n)}})XDYE_j(N_{t_{k(n)}}-N_{s_{k(n)}})\|\ge a^2/2
  \]
  for all $n$.
  Now, for almost any $x\in K$, given any open interval $(s,t)$ which contains
  $x$, and any $n_0\in\NN$, we can find $n\ge n_0$ such that
  $(s_n, t_n)\subseteq(s,t)$. Then, since $n_0\le n\le k(n)\in S_n$
  we know $(s_{k(n)}, t_{k(n)})\subseteq(s,t)$ and $M^\perp_{k(n)}\le M^\perp_{n_0}$
  so that also
  \[
    \|M^\perp_{n_0}(N_t-N_s)XDYE_j(N_t-N_s)\|\ge a^2/2
  \]
  Thus $c_{j, x}(XDY)\ge a^2/2$ for almost every $x\in K$ and yet, since
  $D\in E_i\alg\N E_i\subseteq\T$, we have $XDY\in\T$,
  contradicting the fact that $K$ is disjoint from $C_j$. So
  Property~(\ref{definition-of-extended-triangular-system--column})
  must hold.
\end{proof}

\begin{theorem}\label{contains-marginally-zero-operators-thm}
  Let $\T$ be a maximal triangular algebra satisfying
  \[
    \alg{\N_0} \otimes \D_0 \subseteq \T \subseteq \alg{\N_0}\otimes B(\K)
  \]
  and let $(\bfS, \bfR,\bfC)$ be a nearly triangular system for $\T$.
  Then $\T = \tsrc$ and $(\bfS, \bfR,\bfC)$ is an extended triangular system
  if and only if
  $\T$ contains all $X\in\T(\bfS)$ such that
  for each $m\in\NN$ there is an $n\in\NN$ such that
  $E_m X M_n^\perp = 0$ and $M_n^\perp X E_m=0$.
\end{theorem}

\begin{proof}
  Necessity is trivial, since all $X\in\T(\bfS)$ which satisfy the condition
  must satisfy $r_{i,t}(X)=c_{j,t}(X)=0$ for all $t$. We focus now
  on the converse.
  By the maximality of $\T$ it suffices, in view of
  Theorem~\ref{extended-triangular-systems-make-tsrc-triangular} and
  Proposition~\ref{triangular-contained-in-ts},
  to show that in this case
  property~(\ref{definition-of-extended-triangular-system--row-column})
  of Definition~\ref{definition-of-extended-triangular-system} also holds.
  Suppose for a contradiction that
  property~(\ref{definition-of-extended-triangular-system--row-column})
  does not hold. Then for some $i$ and $j$, $R_i\cap C_j \setminus S_{i,j}$
  is non-null. Thus we can find operators $X,Y\in\T$ and $a>0$ such that
  $r_{i,t}(X) > a$ and $c_{j,t}(Y) > a$ on a non-null set $K$ which is disjoint
  from $S_{i,j}$.
  Let the intervals $(s_n, t_n)$ and sets $S_n$ be chosen as in
  Remark~\ref{interval-family-remark}.

  For each fixed $n$, pick
  $s_n<x<y<z<t_n$ where $x$ and $z$ are in $K$. Let
  $A_n := E_i (N_{t_n}-N_{s_n})X$ and
  $B_n := Y(N_{t_n}-N_{s_n})E_j$.
  Clearly $A_n$ and $B_n$ converge strongly to zero.
  Also,
  $\|A_n M_n^\perp N_y\| \ge \| E_i (N_y-N_{s_n})X(N_y-N_{s_n}) M_n^\perp \| \ge r_{i,x}(X) > a$
  and, similarly
  $\|M_n^\perp N_y^\perp B_n\| > a$.
  Thus we can find a finite-rank
  contraction $D_n$ such that $\|A_n D_n B_n\| > a^2$, which satisfies
  $D_n = M_n^\perp N_y D_n N_y^\perp M_n^\perp$, and consequently belongs to $\alg\N$.
  By weak lower-semicontinuity of the norm, we can also stipulate
  that $D_n = M_k D_n M_k$ for some sufficiently large $k$.

  By Lemma~\ref{linking-lemma}, there is a subsequence $k(n)\in S_n$ such that
  $D := \sum_n D_{k(n)}$ converges strongly and
  \[
    \|
      E_i(N_{t_{k(n)}}-N_{s_{k(n)}})X
      D
      Y (N_{t_{k(n)}}-N_{s_{k(n)}}) E_j
    \| > a^2/2.
  \]
  Thus by Remark~\ref{interval-family-remark},
  for almost every $x\in K$, if the open interval $(s, t)$ contains $x$
  then there is an $n$ such that $(s_{k(n)}, t_{k(n)})\subseteq (s,t)$
  and so
  $\|(N_t-N_s) E_iXDYE_j (N_t-N_s)\| > a^2/2$.
  Thus $i_x(E_i XDY E_j) \ge a^2/2$ for almost every $x\in K$.
  Furthermore, $D$ is in $\T$ since for each $m$, $D E_m$ and
  $E_m D$ are finite sums of $D_kE_m$'s and $E_mD_k$'s respectively.
  So for sufficiently large $n$,
  $D E_m = M_n D E_m$ and $E_m D = E_m D M_n$. Since also each
  $D E_m$ and $E_m D$ are finite-rank, $D$ is in $\T(\bfS)$ and hence
  also in $\T$ by hypothesis.

  The result follows, since the fact $XDY\in\T$ and $i_x(E_i XDY E_j) \ge a^2/2$
  almost everywhere on $K$ together contradict the assumption that $K$ is
  disjoint from $S_{i,j}$.
\end{proof}

One way to interpret the last result is to start with a maximal triangular
algebra $\T$ satisfying the inclusion relation and first calculate
the triangular system $\bfS$ such that $\T\subseteq\T(\bfS)$ and then form
the extended triangular system $(\bfS,\bfR_0,\bfC_0)$
where $\bfR_0=\bfC_0=\emptyset$. Then $\T(\bfS,\bfR_0,\bfC_0)$ is
a triangular algebra contained in $\T(\bfS)$. We have just seen that
$\T$ is simple if and only if it contains $\T(\bfS,\bfR_0,\bfC_0)$.

The following example is of a maximal triangular which is \emph{not} simple,
and illustrates one way simplicity can fail; if the algebra has
different asymptotic behavior at infinity on different infinite subsets
of the indexing set.

\begin{example}\label{non-simple-uniform-alegbra-example}
  Let $E_i$ be indexed by $i\in\ZZ$ and $M_n := \sum_{m\le n}E_m$.
  By Theorem~\ref{tensor-product-construction-citation} the set, $\T$, of
  $X\in\alg\N$ such that $M_n^\perp XM_n\in\rinfty$ for all $n$
  is a maximal triangular algebra.
  Note
  $R_i = C_j = [0,1]$ for all $i,j$ and so $R_i\cap C_j\not\subseteq S_{i,j}$
  for $i > j$, and $(\bfS, \bfR, \bfC)$ is not an extended triangular system.
  Note also however that every $X\in\T$ satisfies
  $r_{i,x}(XM_0) = c_{j,x}(M_0^\perp X) = 0$ almost everywhere, so that
  the support sets along rows and columns are very different when localized
  to this projection $M_0$ or to its complement:
  \[
    \bigvee_{X\in\T, a>0} \{ t \st r_{i,t}(XM_0) \ge a\} \not=
    \bigvee_{X\in\T, a>0} \{ t \st r_{i,t}(XM_0^\perp) \ge a\}
  \]
  and
  \[
    \bigvee_{X\in\T, a>0} \{ t \st c_{j,t}(M_0^\perp X)\ge a \} \not=
    \bigvee_{X\in\T, a>0} \{ t \st c_{j,t}(M_0X)\ge a \}
  \]
\end{example}

In the following result we establish the converse: if the diagonal seminorms
asymptotically have the same support sets when localized to any infinite set
along the rows and columns then the algebra is a simple uniform algebra.

\begin{definition}
  For any $S\subseteq\NN$ write $M_S := \sum_{i\in S} E_i$ and define
  \[
    r^\infty_{i, t}(X) :=
      \inf\{ i_t(E_iXM_S) \st S\subseteq\NN \text{ is infinite } \}
  \]
  and
  \[
    c^\infty_{j, t}(X) :=
      \inf\{ i_t(M_SXE_j) \st S\subseteq\NN \text{ is infinite } \}
  \]
  Then for each $i,j\in\NN$ let
  \begin{align*}
    R^\infty_{i}   &:= \essunion_{X\in\T, a>0} \{t \st r^\infty_{i,t}(X) \ge a \} \\
    C^\infty_{j}   &:= \essunion_{X\in\T, a>0} \{t \st c^\infty_{j,t}(X) \ge a \}
  \end{align*}
\end{definition}

\begin{theorem}\label{infinite-subsets-algebra-thm}
  Let $\T$ be a triangular algebra satisfying
  \[
    \alg{\N_0} \otimes \D_0 \subseteq \T \subseteq \alg{\N_0}\otimes B(\K)
  \]
  and let $(\bfS, \bfR,\bfC)$ be a nearly triangular system for $\T$.
  Then $\T=\tsrc$ and $(\bfS, \bfR,\bfC)$ is an extended triangular system
  if and only if
  \[
    C_j = C^\infty_j \text{ and }
    R_i = R^\infty_i
  \]
  to within a null set for all $i$ and $j$.
\end{theorem}

This theorem relies on the main result of \cite{Orr:EsNoPrInBlOpMa} in which we used
techniques of Infinite Ramsey Theory to prove the following:

\begin{theorem}[\cite{Orr:EsNoPrInBlOpMa}, Theorem 1.2]\label{infinite-subseteq-estimate-thm}
  Let $X,Y\in\bh$ and suppose that $\|XM_S\| > 1$ and $\|M_S Y\| > 1$
  for all infinite $S\subseteq\NN$. Then there is a block diagonal contraction
  $D$ such that $\|XDY\| \ge 1/5$.
\end{theorem}

\begin{proof}[Proof of Theorem~\ref{infinite-subsets-algebra-thm}]
  We shall first prove necessity, so suppose $\T=\tsrc$.
  Clearly $C_j \supseteq C^\infty_j$ and
  $R_i \supseteq R^\infty_i$ for all $i$ and $j$. Suppose if possible that
  there is a non-null closed set $K\subseteq C_j\setminus C^\infty_j$.
  Find a sequence of countably many pairwise disjoint measure-dense
  subsets of $[0,1]$ (see \cite[Lemma 3.1]{Orr:TrAlIdNeAl} for a construction)
  and index them as $F_{m,n}$ for $m,n\in\NN$. Let $(s_n, t_n)$ be an enumeration
  of all the intervals with rational endpoints which contain a point of $K$.
  For each fixed $m$, let $n$ run through $\NN$ and pick a rank-1 operator $R_{m,n}$
  of unit norm satisfying
  \[
    R_{m,n} =
      E_m (N_x - N_{s_n})E(F_{m,n}) R_{m,n} E(F_{m,n})(N_{t_n} - N_x) E_j
  \]
  for some $x\in K$ with $s_n<x<t_n$. Let
  $T := \sum_{m=1}^\infty\sum_{n=1}^m R_{m,n}$, which converges strongly since
  the ranges and domains of the $R_{m,n}$ are pairwise orthogonal.
  Clearly $T=T E_j\in\alg\N$ and for any $m$, $E_mTE_j$ is finite rank, so $i_x(E_m TE_n)=0$
  for all $m,n$ and all $x\in[0,1]$. Likewise $r_{m,x}(T)\le i_x(E_mTE_j)=0$ for all
  $m$ and $x$.
  If $x\not\in K$ then there is $s<x<t$ such that $(s,t)$ is disjoint from $K$
  and so $(N_t-N_s)R_{m,n}(N_t-N_s)=0$ for all $m,n$ and hence $(N_t-N_s)T(N_t-N_s)=0$.
  Thus $c_{j,x}(T)\le i_x(TE_j)=0$ for $x\not\in K$, so that $T\in\tsrc=\T$.
  However for any fixed $x\in K$ and $s<x<t$, find $n$ such that
  $(s_n, t_n)\subseteq(s, t)$ and observe that for any $m>n$
  \begin{align*}
    \|E_m(N_t-N_s) & T (N_t-N_s)E_j\| \\
      & \ge \|E_m (N_{t_n} - N_{s_n})E(F_{m,n}) T E(F_{m,n})(N_{t_n} - N_{s_n}) E_j\| \\
      & \ge \|E_m (N_{t_n} - N_{s_n})E(F_{m,n}) R_{m,n} E(F_{m,n})(N_{t_n} - N_{s_n}) E_j\| \\
      & = \|R_{m,n}\| = 1
  \end{align*}
  Thus if $S\subseteq\NN$ is infinite then
  $\|M_S(N_t-N_s) T (N_t-N_s)E_j\| \ge 1$ and so $c^\infty_{j,x}(T)\ge 1$ on $K$,
  contradicting the assumption that $K$ was
  disjoint from $C^\infty_j$. It follows by contradiction that $C_j = C^\infty_j$ and
  the fact that $R_i = R^\infty_i$ follows similarly.

  We now prove sufficiency. As in Theorem~\ref{contains-marginally-zero-operators-thm},
  it is enough to
  prove that property~(\ref{definition-of-extended-triangular-system--row-column})
  of Definition~\ref{definition-of-extended-triangular-system} holds.
  Suppose for a contradiction that
  property~(\ref{definition-of-extended-triangular-system--row-column})
  does not hold. Then for some $i$ and $j$, $R^\infty_i\cap C^\infty_j \setminus S_{i,j}$
  is non-null. Thus we can find operators $X,Y\in\T$ and $a>0$ such that
  $r^\infty_{i,t}(X) > a$ and $c^\infty_{j,t}(Y) > a$ on a non-null set $K$ which is disjoint
  from $S_{i,j}$.
  Let the intervals $(s_n, t_n)$ and sets $S_n$ be chosen as in
  Remark~\ref{interval-family-remark}.

  As usual, for each fixed $n$, pick
  $s_n<x<y<z<t_n$ where $x$ and $z$ are in $K$. Let
  $A_n := E_i (N_{t_n}-N_{s_n})X M_n^\perp$ and
  $B_n := M_n^\perp Y(N_{t_n}-N_{s_n})E_j$.
  Clearly $A_n$ and $B_n$ converge strongly to zero.
  Also, for any infinite $S\subseteq\NN$,
  \[
    \|A_n N_y M_S\|
      \ge \| E_i (N_y-N_{s_n})X(N_y-N_{s_n}) M_{S\cap (n, \infty)} \|
      \ge r^\infty_{i,x}(X)
      > a
  \]
  and, similarly,
  $\|M_S N_y^\perp B_n\| > a$.
  Thus, by Theorem~\ref{infinite-subseteq-estimate-thm} there is an infinite
  block diagonal contraction  $D_n$ such that
  $\|A_n N_y D_n N_y^\perp B_n\| \ge a' := a^2/5$.
  Finite rank operators are weakly dense in the set of infinite block diagonals
  so by weak lower semicontinuity of the norm, $D_n$ can be assumed to be finite
  rank.
  Without loss, also take $D_n=M_n^\perp N_y D_n N_y^\perp M_n^\perp$ so that
  $D_n$ is in $\alg{\N_0}\otimes\D_0$ and converges strong-* to zero.

  The proof now completes exactly as
  Theorem~\ref{contains-marginally-zero-operators-thm}.
  By Lemma~\ref{linking-lemma},
  there is a subsequence $k(n)\in S_n$ such that $D := \sum_n D_{k(n)}$ converges
  strongly and
  \[
    \|
      E_i(N_{t_{k(n)}}-N_{s_{k(n)}})X
      D
      Y (N_{t_{k(n)}}-N_{s_{k(n)}}) E_j
    \| > a'/2.
  \]
  Thus by Remark~\ref{interval-family-remark},
  for almost every $x\in K$, if the open interval $(s, t)$ contains $x$
  then there is an $n$ such that $(s_{k(n)}, t_{k(n)})\subseteq (s,t)$
  and so
  $\|(N_t-N_s) E_iXDYE_j (N_t-N_s)\| > a'/2$.
  Thus $i_x(E_i XDY E_j) \ge a'/2$ for almost every $x\in K$.
  Furthermore, since each $D_n$ is in $\alg{\N_0}\otimes\D_0$
  which is weakly closed, it follows that
  $D\in\alg{\N_0}\otimes\D_0\subseteq\T$.

  The result then follows, since the fact $XDY\in\T$ and $i_x(E_i XDY E_j) \ge a/2$
  almost everywhere on $K$ together contradict the assumption that $K$ is
  disjoint from $S_{i,j}$.

\end{proof}

%
%


\bibliography{bibliography}
\bibliographystyle{plain}

\end{document}

%% file: macros.tex
\newtheorem{theorem}{Theorem}[section]
\newtheorem{proposition}[theorem]{Proposition}
\newtheorem{lemma}[theorem]{Lemma}

\theoremstyle{definition}
\newtheorem{definition}[theorem]{Definition}

\newtheorem{remark}[theorem]{Remark}
\newtheorem{example}[theorem]{Example}

\newcommand{\e}{\epsilon}
\newcommand{\A}{\mathcal{A}}
\newcommand{\B}{\mathcal{B}}
\newcommand{\CC}{\mathbb{C}} 
\newcommand{\NN}{\mathbb{N}} 
\newcommand{\QQ}{\mathbb{Q}} 
\newcommand{\RR}{\mathbb{R}} 
\newcommand{\ZZ}{\mathbb{Z}} 

\renewcommand{\H}{\mathcal{H}} 
\newcommand{\K}{\mathcal{K}} 
\newcommand{\D}{\mathcal{D}}
\newcommand{\M}{\mathcal{M}} 
\newcommand{\N}{\mathcal{N}} 

\newcommand{\R}{\mathcal{R}}
\newcommand{\T}{\mathcal{T}} 
\newcommand{\bh}{B(\H)} 
\newcommand{\alg}[1]{\mathop{Alg}(#1)}

\newcommand\eqae{\stackrel{a.e.}{=}}

\newcommand{\essunion}{\bigvee}
\newcommand{\rinfty}{\R^\infty_\N}

\newcommand{\masa}{masa}

\newcommand{\st}{\;:\;} 
\newcommand{\bfR}{\textbf{R}}
\newcommand{\bfC}{\textbf{C}}
\newcommand{\bfS}{\textbf{S}}

\newcommand{\tsrc}{\T(\bfS, \bfR, \bfC)}